\documentclass[a4paper]{amsart}
\usepackage{amssymb,amsfonts}
\usepackage{amsthm}
\usepackage[all,arc]{xy}
\usepackage{enumerate}
\usepackage{mathrsfs}
\usepackage{tikz}
\usepackage{url}
\usetikzlibrary{positioning}
\usetikzlibrary{matrix,arrows,decorations.pathmorphing}
\usepackage{tikz-cd}
\usepackage{fancyhdr}
\tikzset{commutative diagrams/.cd}
\usetikzlibrary{arrows,snakes,backgrounds}
\tikzset{>=stealth}
\usepackage{color}   
\usepackage{marginnote}
\usepackage{nameref}
\usepackage[backref=page]{hyperref}
\usepackage{MnSymbol} 
\hypersetup{
    colorlinks=true, 
    linktoc=all,     
    citecolor=black,
    filecolor=black,
    linkcolor=blue,
    urlcolor=black
}
\usepackage{cleveref}
\usetikzlibrary{arrows}

\newtheorem{thm}{Theorem}[section]
\newtheorem{cor}[thm]{Corollary}
\newtheorem{prop}[thm]{Proposition}
\newtheorem{lem}[thm]{Lemma}

\newtheorem{quest}[thm]{Question}
\newtheorem*{thm*}{Theorem}

\theoremstyle{definition}
\newtheorem{defn}[thm]{Definition}

\theoremstyle{remark}
\newtheorem{rem}[thm]{Remark}

\newcommand{\bD}{\mathbb{D}}

\newcommand{\bF}{\mathbb{F}}

\newcommand{\bQ}{\mathbb{Q}}
\newcommand{\bR}{\mathbb{R}}
\newcommand{\bS}{\mathbb{S}}

\newcommand{\bZ}{\mathbb{Z}}

\newcommand\Diff{\mathrm{Diff}}
\newcommand\BDiff{\mathrm{BDiff}}
\newcommand\dDiff{\mathrm{Diff}^{\delta}}

\newcommand\BdDiff{\mathrm{BDiff}^{\delta}}

\newcommand{\hcoker}{/\!\!/}

\addtocontents{toc}{\protect\setcounter{tocdepth}{1}}
\makeatletter
\let\c@equation\c@thm
\makeatother
\numberwithin{equation}{section}
\title{Braid groups and discrete diffeomorphisms of the punctured disk}

\author{Sam Nariman}

\email{sam@math.northwestern.edu}
\address{Department of Mathematics\\
  Northwestern University\\
2033 Sheridan Road\\
Evanston, IL  60208}

\begin{document}
\begin{abstract}
We show that the group cohomology of the diffeomorphisms of the  disk with $n$ punctures has the cohomology of the braid group of $n$ strands as the summand.  As an application of this method, we also prove that there is no cohomological obstruction to lifting the ``standard" embedding $\mathrm{Br}_{2g+2}\hookrightarrow \mathrm{Mod}_{g,2}$ to a group homomorphism between diffeomorphism groups. 

\end{abstract}
\maketitle
\section{Statement of the results}\label{braids}
Let $S$ be a surface possibly with boundary and let ${\bf z}\subset S$ be a subset of $|{\bf z}|=n$ points on the surface. We shall write $\Diff(S -{\bf z},\partial S)$ for the topological group (equipped with the $C^{\infty}$-topology) of the orientation preserving smooth diffeomorphisms of the punctured surface $S -{\bf z}$ whose supports are away from the boundary $\partial S$.  To study certain algebraic properties of this group, we also consider the same group $\dDiff(S -{\bf z},\partial S)$ but with the discrete topology. Recall the mapping class group $\mathrm{Mod}(S,{\bf z})$ is the group of connected components of the topological group $\Diff(S -{\bf z},\partial S)$ and there are natural maps
\[
\dDiff(S -{\bf z},\partial S)\to \Diff(S -{\bf z},\partial S)\to \mathrm{Mod}(S,{\bf z}),
\]
where the first map is the identity homomorphism and the second is taking quotient by the identity component of the group $\Diff(S -{\bf z},\partial S)$.

The {\it realization problem} for a subgroup $H\hookrightarrow  \mathrm{Mod}(S,{\bf z})$ is concerned with whether one can lift $H$ to $\dDiff(S -{\bf z},\partial S)$ as a subgroup.  One of the positive results in this direction is the solution to the Nielsen realization problem by Kerckhoff \cite{kerckhoff1983nielsen} who proved that any finite subgroup of the mapping class group of a surface can be realized by diffeomorphisms. On the other hand, for surfaces with no punctures Morita \cite{morita1987characteristic}  used the Bott vanishing theorem to show that the induced map between group cohomologies
\[
H^*(\mathrm{Mod}(S);\bQ)\to H^*(\dDiff(S ,\partial S);\bQ),
\]
has a kernel while the genus $g(S)$ is larger than $10$ and he concluded that  finite index subgroups of the mapping class group of the surface $S$ cannot be realized as a subgroup of the diffeomorphism group of the surface. 

\subsection{Splitting of the cohomology} For the case of the disk $S=\bD^2$, Salter-Tshishiku \cite{salter2015nonrealizability} showed that the braid group $\mathrm{Mod}(\bD^2, {\bf z})$ cannot be realized as a subgroup of $\dDiff(\bD^2-{\bf z},\partial \bD^2)$ for $|{\bf z}|\geq 5$ using dynamical system techniques but unlike Morita's theorem we will show that
\begin{thm}\label{punctures}
The map
\[
H^*(\mathrm{Mod}(\bD^2, {\bf z});A)\to H^*(\dDiff(\bD^2-{\bf z},\partial\bD^2);A),
\]
is split injective in all cohomological degrees and for all abelian groups $A$.
\end{thm}
\begin{rem}
For $|z|\geq 5$, the proof of the above theorem implies a slightly stronger result that the induced map between the plus constructions of the classifying spaces 
\[
\BdDiff(\bD^2-{\bf z},\partial\bD^2)^+\to \mathrm{BMod}(\bD^2, {\bf z})^+,
\]
admits a section, where $+$ means Quillen plus construction at the commutator subgroups.
\end{rem}
\begin{rem}
For homeomorphism groups, Thurston observed in \cite{mathoverflow} that for $|{\bf z}|=3$,
\[
\mathrm{Homeo}(\bD^2-{\bf z},\partial\bD^2)\rightarrow \mathrm{Mod}(\bD^2, {\bf z}),
\]
admits a section.
\end{rem}
The situation of realizing braid groups by diffeomorphisms of the punctured disk is similar to that of realizing the mapping class group of a surface of genus $g>5$ by the surface homeomorphism group. That is to say, in that case the mapping class group and the homeomorphism group have the same homology (see \cite{mcduff1980homology}) but still there is no section from the mapping class group of such a surface to its homeomorphism group (see \cite{markovic2007realization}). But the difference is, unlike the case of surface homeomorphisms, the group homology of $\dDiff(\bD^2-{\bf z},\partial\bD^2)$ is much bigger than the homology of $\mathrm{Mod}(\bD^2, {\bf z})$. 

In fact a more general result holds for any surface $S$ which specializes to \Cref{punctures} for $S=\bD^2$. Let $C_n(S)$ be the configuration space of $n$ unordered points on the surface $S$. The surface braid group $\mathrm{Br}_n(S)$ is the fundamental group of the space $C_n(S)$. Let $\Diff(S,{\bf z})$ denote the orientation preserving diffeomorphisms of $S$ that fix the marked points ${\bf z}$ as a set and are the identity near the boundary of $S$ and let $\mathrm{Mod}(S,{\bf z})$ be the mapping class group of the surface $S$ with $n$ marked points ${\bf z}$ in the interior of $ S$. There is a so called {\it point-pushing} map
\[
\mathcal{P}: \mathrm{Br}_{|{\bf z}|}(S)\rightarrow \mathrm{Mod}(S,{\bf z})
\]
that is induced by the long exact sequence of homotopy groups for the fibration $\Diff(S,{\bf z})\rightarrow \Diff(S, \partial S)\rightarrow C_{|{\bf z}|}(S)$. Nick Salter and Bena Tshishiku \cite{salter2015nonrealizability} proved that for a surface $S$ of genus $g$ and $b$ boundary components if $g+2b\geq 2$ and $|{\bf z}|\geq 5$, the point-pushing map has no lift to the diffeomorphisms of the punctured surface. But we show that there is no cohomological obstruction to realize the surface braid group as diffeomorphisms of the punctured surface:
\begin{thm}\label{surfaces}
For an orientable surface $S$ which is not homeomorphic to sphere, there exists a lift $\alpha^*$ for $\mathcal{P}^*$ in the diagram
 \[
 \begin{tikzpicture}[node distance=1.8cm, auto]
  \node (A) {$H^*(\mathrm{Br}_{|{\bf z}|}(S);A)$};
  \node (B) [right of=A, node distance=3.8cm] {$H^*(\mathrm{Mod}(S,{\bf z});A),$};
  \node (C) [above of= B ] {$H^*(\dDiff(S-{\bf z},\partial S);A)$};  
   \draw [<-] (A) to node {$\mathcal{P}^*$}(B);
  \draw [<-] (A) to node {$\alpha^*$}(C);
  \draw [<-] (C) to node {$$}(B);
\end{tikzpicture}
\]
for all cohomological degrees and for all abelian groups $A$.
\end{thm}
For $S=\bD^2$, we know that $\Diff(\bD^2,\partial \bD^2)$ is contractible (\cite{MR0112149}). Therefore the long exact sequence of the homotopy groups for the fibration $\Diff(\bD^2,{\bf z})\rightarrow \Diff(\bD^2, \partial \bD^2)\rightarrow C_{|{\bf z}|}(\bD^2)$ implies that $\mathcal{P}:\mathrm{Br}_{|{\bf z}|}(\bD^2)\to \mathrm{Mod}(\bD^2, {\bf z})$ is an isomorphism. Thus  \Cref{punctures} is a special case of  \Cref{surfaces} for $S=\bD^2$.

\subsection{Geometric meaning of the splitting} Using bordism theory, we can geometrically interpret the splitting  \Cref{punctures} as follows. Let a {\it $B_k$-object}\footnote{We borrowed this notation from \cite{fuks1974quillenization}.} be a triple $(M, N, \epsilon)$, where $M$ is an oriented compact smooth manifold, $N$ is an oriented smooth compact submanifold of the product $M \times \bR^k$ of dimension equal to the dimension of the manifold $M$, and $\epsilon$ is a trivialization of the normal bundle to $N$ in $M\times \bR^k$. If $N$ has a boundary, we assume that $\partial N\subset \partial M\times \bR^k$ and $\epsilon$ restricts to a trivialization of the normal bundle of $\partial N$ inside $\partial M\times \bR^k$. Thus we can define the boundary of the $B_k$-object $(M,N, \epsilon)$ to be $(\partial M, \partial N, \epsilon')$ where $\epsilon'$ is induced by the restriction of $\epsilon$. Given two $B_k$-objects $(M_1,N_1,\epsilon)$ and $(M_2,N_2,\epsilon_2)$ without boundary, we say they are bordant if there exists another $B_k$-object whose boundary is $(M_1\coprod -M_2,N_1\coprod -N_2,\epsilon_1\coprod \epsilon_2)$. We say $(M, N, \epsilon)$ is nonsingular if the restriction of the projection map $ M\times \bR^k \to M$ to $N$ is a regular map. 

For manifolds $M$ and $L$, the trivial bundle $M\times L\to M$ admits a horizontal foliation which is given by the leaves $M\times \{ x\}$ for all $x\in L$. Let $\mathcal{F}$ be a foliation on $M\times L$ transverse to the fiber of $M\times L\to M$. We say $\mathcal{F}$ is compactly supported if there exits a compact subset $K\subset L$ such that the restriction of $\mathcal{F}$ to $M\times (L\backslash K)$ coincides with the restriction of the horizontal foliation to $M\times (L\backslash K)$.
 \begin{thm}
 For $k=2$ and $k=3$, every nonsingular $B_k$-object $(M, N, \epsilon)$ is bordant to another nonsingular $B_k$-object $(M', N', \epsilon')$ where the bundle $p: M'\times \bR^k \backslash N'\to M'$ admits a compactly supported foliation that is transverse to the fibers of the projection $p$.
 \end{thm}
 \subsection{Inducing up the map $\mathrm{BBr}_{2g+2}(\bD^2)\to \mathrm{BMod}(\Sigma_{g,2})$} 
  There are many ways to embed the braid group into the mapping class group of a surface. We are interested in the embedding induced by the geometric description given by Tillmann and Segal in \cite{segal2007mapping}. To recall the map, let ${C}_{2g+2}(\bD^2)$ be the configuration of  $2g+2$ unordered  points in interior of the unit disk and let $\mathcal{M}_{g,2}$ be the moduli
space of connected Riemann surfaces of genus $g$ with two ordered and parametrized boundary components. There is a map
\[
\phi: {C}_{2g+2}(\bD^2)\to \mathcal{M}_{g,2}
\]
 which sends $2g+2$ points ${\bf z}= \{z_1,z_2,\dots, z_{2g+2}\}$ to the Riemann surface $\Sigma_{{\bf z}}$ given by 
 \[
 f_{{\bf z}}(z)^2 =\prod_{i=1}^{2g+2} (z-z_i).
 \]
This surface which has genus $g$ and $2$ boundary components is a branch cover over the unit disk with ${\bf z}$ as the set of branched points. We consider the map that is induced on the fundamental groups
 \[
 \psi: \mathrm{Br}_{2g+2}(\bD^2)\to \mathrm{Mod}(\Sigma_{g,2}).
 \]
 Birman and Hilden \cite{birman1973isotopies} proved that this map is injective. The main theorem in \cite{song2007braids} and \cite{segal2007mapping} is that $\psi$ induces the trivial map on the stable homology.
 
 Motivated by the Lie theoretic version of the Margulis Superrigidity, which asserts that homomorphisms between lattices (virtually) extend to homomorphisms of the ambient groups, Aramayona and Suoto in \cite{aramayona2012rigidity} asked if the same is true for group homomorphisms between mapping class groups. In our case, the question becomes: 
 \begin{quest}\label{quest}
 Is it true that the group homomorphism $\psi$ is (virtually) induced by a homomorphism $\Psi$ between diffeomorphism groups? In other words, does there exist  a homomorphism $\Psi$ that makes the following diagram commute?
  \[
 \begin{tikzpicture}[node distance=4.6cm, auto]
  \node (A) {$ \dDiff(\bD^2, (2g+2)\text{ marked points})$};
  \node (B) [right of=A] {$\dDiff(\Sigma_{g,2},\partial\Sigma_{g,2})$};
  \node (C) [below of= A, node distance=1.8cm ] {$\mathrm{Br}_{2g+2}(\bD^2)$};  
  \node (D) [right of=C] {$\mathrm{Mod}(\Sigma_{g,2}).$};
   \draw [->] (A) to node {$\Psi$}(B);
  \draw [->] (C) to node {$\psi$}(D);
  \draw [->] (A) to node {$$}(C);
  \draw [->] (B) to node {$$} (D);
\end{tikzpicture}
\]
 \end{quest}

 As we shall see in \Cref{induceup}, one can use covering space theory to  lift $\psi$ to a group homomorphism between $\dDiff(\bD^2, (2g+2)\text{ marked points})$ and the homeomorphism group $\mathrm{Homeo}^{\delta}(\Sigma_{g,2},\partial\Sigma_{g,2})$ but it is not hard to see that such a lift cannot be made differentiable at the ramification points in $\Sigma_{g,2}$. 
 

We show that there is no homological obstruction for $\Psi$ to exist in the following sense:
 \begin{thm}
There exists a map $\Phi$ which makes the following diagram homotopy commutative
   \[
 \begin{tikzpicture}[node distance=5cm, auto]
  \node (A) {$ \BdDiff(\bD^2, (2g+2)\text{ marked points})$};
  \node (B) [right of=A] {$Y$};
  \node (C) [below of= A, node distance=1.8cm ] {$\mathrm{BBr}_{2g+2}(\bD^2)$};  
  \node (D) [right of=C] {$\mathrm{BMod}(\Sigma_{g,2}).$};
   \draw [->] (A) to node {$\Phi$}(B);
  \draw [->] (C) to node {$\phi$}(D);
  \draw [->] (A) to node {$$}(C);
  \draw [->] (B) to node {$$} (D);
\end{tikzpicture}
\]
where $Y$ is a space over $\mathrm{BMod}(\Sigma_{g,2})$ that is only homology equivalent to the classifying space $\BdDiff(\Sigma_{g,2},\partial\Sigma_{g,2})$.
 \end{thm}

\begin{rem}
Kathryn Mann (\cite{Katie}) recently found a construction  for the map $\Psi$ which solves the problem \ref{quest}.
\end{rem}
 \begin{rem}
Unlike the main theorem in \cite{song2007braids} and \cite{segal2007mapping} which says that $\phi$ induces a trivial map in the stable homology, we show that $\Phi$ induces a nontrivial map
 \[
  H_3( \dDiff(\bD^2, (2g+2)\text{ marked points});\bZ)\to H_3(\dDiff(\Sigma_{g,2},\partial\Sigma_{g,2});\bZ)
 \]
between third homologies, using Godbillon-Vey classes.
 \end{rem}
 \section*{Acknowledgments} I would like to thank S\o ren Galatius for his support and helpful discussions during this project. I would like to thank Kathryn Mann for introducing me to the realization problems and her interest in \Cref{quest}. I am indebted to Alexander Kupers for reading the first draft of this work. I also like to thank Jeremy Miller, Johannes Ebert, Ricardo Andrade and Jonathan Bowden for helpful discussions. Bena Tshishiku and Nick Salter provided me with an earlier draft of their paper for which I am very grateful. This project was partially supported by NSF grant DMS-1405001. Finally I would like thank the referee for his/her careful reading and many helpful suggestions.

\section{Splitting of the group $H^*( \dDiff(\bD^2-{\bf z},\partial \bD^2);\bZ)$}\label{sec:2}
In this section, we define the notion of {\it configuration bundle maps} inspired by the work of Ellenberg, Venkatesh and Westerland in \cite{ellenberg2012homological} to prove \Cref{punctures}. Using configuration bundle maps, we give a model for the plus construction of $\BdDiff( \bD^2-{\bf z},\partial \bD^2)$ that also maps to $\mathrm{BMod}(\bD^2, {\bf z})$. Given this model, we then find a space level section which leads to a section between (co)homology groups. 

\subsection{Recollection from foliation theory}\label{foliation} Let us first briefly recall Mather-Thurston theory (see \cite[Section 1.2.1]{nariman2015stable} or \cite[Section 5.1]{nariman2014homologicalstability} for more details). 

Let $\mathrm{S}\Gamma_n$ denote the topological category whose objects are points in $\bR^n$ with the usual topology and whose morphisms are germs of orientation preserving diffeomorphisms of $\bR^n$. Given an open cover $\mathcal{U}=\{ U_i\hookrightarrow X\}$ on a topological space $X$, we can define the {\it \v{C}ech groupoid} $X_{\mathcal{U}}$ as follows. The space of objects of $X_{\mathcal{U}}$ is given by the disjoint union $\coprod_i U_i$ and the space of morphisms is given by the disjoint union $\coprod_{i,j} U_i\cap U_j$. An $\mathrm{S}\Gamma_n$-cocycle  on a topological space $X$ is given by a covering $\mathcal{U}$ of $X$ and a functor from the groupoid $X_{\mathcal{U}}$ to the groupoid $\mathrm{S}\Gamma_n$. Two $\mathrm{S}\Gamma_n$-cocycles, $F_1:X_{\mathcal{U}}\to \mathrm{S}\Gamma_n$ and $F_2:X_{\mathcal{V}}\to \mathrm{S}\Gamma_n$ are equivalent if there exists a functor $F:X_{\mathcal{U}\cup \mathcal{V}}\to \mathrm{S}\Gamma_n$ such that $F|_{X_{\mathcal{U}}}=F_1$ and $F|_{X_{\mathcal{V}}}=F_2$.

A $\mathrm{S}\Gamma_n$-structure on $X$ is an equivalent class of $\mathrm{S}\Gamma_n$-cocycles. Two $\mathrm{S}\Gamma_n$-structures $F_1$ and $F_2$ are {\it concordant} if there exists an $\mathrm{S}\Gamma_n$-structure $F$ on $X\times I$ so that $F|_{X\times\{0\}}=F_1$ and $F|_{X\times\{1\}}=F_2$.  Concordance class of  $\mathrm{S}\Gamma_n$-structures on $X$ are in bijection with homotopy classes of maps from $X$ to the classifying space $\mathrm{BS}\Gamma_n$. In particular, on a manifold $M$ any codimension $n$ foliation $\mathcal{F}$ whose normal bundle is oriented, gives rise to an $\mathrm{S}\Gamma_n$-structure on $M$ and therefore the foliation $\mathcal{F}$ induces a map $M\to \mathrm{BS}\Gamma_n$ well defined up to homotopy. 

We can consider the topological group $\mathrm{GL}_n(\bR)^+$ of invertible real $n$ by $n$ matrices with positive determinants as a topological category with one object. There is a functor from $\mathrm{S}\Gamma_n$ to $\mathrm{GL}_n(\bR)^+$ that sends every morphism in $\mathrm{S}\Gamma_n$ to its derivative. This functor induces a map between classifying spaces  $\nu: \mathrm{BS}\Gamma_n\to \mathrm{BGL}_n(\bR)^+$. For a smooth $n$-manifold $M$ possibly with boundary, the foliation by points gives an $\mathrm{S}\Gamma_n$-structure whose normal bundle is the tangent bundle of the manifold $M$. By a general theory of Haefilger \cite[Section 3]{haefliger1971homotopy}, this $\mathrm{S}\Gamma_n$-structure induces a commutative diagram up to homotopy 
 \begin{equation}\label{eq:11}
 \begin{gathered}
 \begin{tikzpicture}[node distance=1.8cm, auto]
  \node (A) {$M$};
  \node (B) [right of=A] {$\mathrm{BGL}_{n}(\bR)^+,$};
  \node (C) [above of= B ] {$\mathrm{BS}\Gamma_{n}$};  
   \draw [->] (A) to node {$\tau$}(B);
  \draw [->] (C) to node {$\nu$}(B);
  \draw [->] (A) to node {$s_0$}(C);
\end{tikzpicture}
\end{gathered}
 \end{equation}
where $\tau$ is the classifying map for the tangent bundle $TM$. Therefore $s_0$ lifts the orientation structure on $TM$ to a tangential structure given by the map $\nu$. Let $\gamma$ be the tautological bundle over $\mathrm{BGL}_{n}(\bR)^+$. Fixing an isomorphism between $TM$ and $\tau^*(\gamma)$, the map $s_0$ induces a bundle map from $TM$ and $\nu^*(\gamma)$. Let $\text{Bun}_{\partial M}(TM,\nu^*(\gamma))$ be the space of bundle maps from $TM$ to $\nu^*(\gamma)$ that are equal to the map induced by $s_0$ on a germ of  a collar of the boundary $\partial M$. The topology on the bundle maps is given by the compact-open topology. Note that the topological group $\Diff(M, \partial M)$ acts on the space $\text{Bun}_{\partial M}(TM,\nu^*(\gamma))$  by precomposing a bundle map with the differential of a diffeomorphism\footnote{For a topological group $G$ acting on a topological space $X$, the homotopy quotient is denoted by $X\hcoker G$ and is given by $X\times_G \mathrm{E}G$ where $\mathrm{E}G$ is a contractible space on which $G$ acts freely.}.

In \cite[Section 1.2.1]{nariman2015stable}, we explained that Mather-Thurston's theorem (\cite{mather2011homology}) implies that there exists a homotopy commutative diagram
\[
 \begin{tikzpicture}[node distance=2.5cm, auto]
  \node (A) {$\BdDiff(M,\partial M)$};
  \node (B) [right of=A, below of=A, node distance=1.7cm]{$\BDiff(M,\partial M),$};
  \node (C) [right of=A, node distance=5cm]{$\text{Bun}_{\partial}(TM,\nu^*(\gamma))\hcoker \Diff(M,\partial M)$};
  \draw [->] (A) to node {$f_M$} (C);
  \draw [<-] (B) to node {$$} (C);
  \draw [->] (A) to node {$$} (B);
 \end{tikzpicture}
\]
where the horizontal map induces a homology isomorphism. In fact we need a relative version of Mather-Thurston's theorem (see  \cite{mcduff1983local} and \cite{mcduff1980homology}) which implies that for a set of points ${\bf z}$ in the interior of $M$, we have a map
\[
f_{M-{\bf z}}:\BdDiff(M-{\bf z},\partial M)\to \text{Bun}_{\partial M}(T(M-{\bf z}),\nu^*(\gamma))\hcoker \Diff(M-{\bf z},\partial M),
\]
that is over $\BDiff(M-{\bf z},\partial M)$ and induces a homology isomorphism. For brevity, we denote the target of the map $f_{M-{\bf z}}$ by $\mathcal{M}_{\nu}(M-{\bf z})$.
\subsection{Configuration bundle maps} For a smooth manifold $M$ with boundary, let $\text{int}(M)$ denote the interior of $M$. Let $[n]$ be the discrete space $\{1,2,\cdots, n\}$. The {\it ordered configuration space} of $n$ points in $M$ is the space 
\[
F_n(M):=\text{Emb}([n],\text{int}(M)).
\]
The symmetric group of $n$ letters $S_n$ acts freely on $F_n(M)$. The space of {\it unordered configuration space} of $n$ points is the quotient space
\[
C_n(M):=F_n(M)/S_n.
\]
Recall from \Cref{braids} that for a set of $n$ points ${\bf z}$ in the interior of $M$ the group $\Diff(M,{\bf z})$ denotes the topological group of $C^{\infty}$-diffeomorphisms of $M$ that fix ${\bf z}$ setwise and are the identity near the boundary $\partial M$.  The space $C_n(M)$ sits in a fibration sequence
\begin{equation}\label{fibration}
 \Diff(M,{\bf z})\to\Diff(M,\partial M)\to C_n(M),
\end{equation}
where the second map is given by the action of $\Diff(M,\partial M)$ on the set ${\bf z}$.
\begin{defn} We define the   {\it configuration bundle maps} $\mathrm{CBun}_k(M)$ to be the space of pairs $$\{({\bf x}, f)|\  { \bf x}\in C_k(M), f\in \mathrm{Bun}_{\partial M}(T(M-{\bf x}),\nu^*(\gamma))\}.$$ To define the topology on these pairs, let ${\bf z}\in \mathrm{C}_k(M)$ be  a fixed configuration of $k$ points in the interior of the manifold $M$. Consider the space
\[
\Diff(M,\partial M)\times_{\Diff(M,{\bf z})}\mathrm{Bun}_{\partial M}(T(M-{\bf z}),\nu^*(\gamma)),
\]
where  the action of $\Diff(M,{\bf z})$ on $\mathrm{Bun}_{\partial M}(T(M-{\bf z}),\nu^*(\gamma))$ is induced by the natural action of $\Diff(M-{\bf z},\partial M)$ on $\mathrm{Bun}_{\partial M}(T(M-{\bf z}),\nu^*(\gamma))$. There is a bijection between this space and  $\mathrm{CBun}_k(M)$ by sending 
\[
(\phi, g)\in \Diff(M,\partial M)\times_{\Diff(M,{\bf z})}\mathrm{Bun}_{\partial M}(T(M-{\bf z}),\nu^*(\gamma)),
\]
to $(\phi({\bf z}), g\circ D(\phi^{-1}))\in \mathrm{CBun}_k(M)$ where $D(\phi^{-1})$ is the derivative of $\phi^{-1}$. This bijection induces a topology on $\mathrm{CBun}_k(M)$ and it is independent of the choice of the fixed ${\bf z}\in C_k(M)$.
 \end{defn}
 The projection map
 \[
\pi: \mathrm{CBun}_k(M)\to C_k(M)
 \]
 that sends the pair $({\bf z},f)$ to the configuration of points ${\bf z}$ is a fibration and the fiber over ${\bf z}$ is space $\mathrm{Bun}_{\partial M}(T(M-{\bf z}),\nu^*(\gamma))$.
 \begin{lem}\label{section}
 The projection $\pi$ has a section.
 \end{lem}
 \begin{proof}
 Recall in the diagram \ref{eq:11}, the map $s_0$ induces a bundle map that lives in $\mathrm{Bun}_{\partial M}(T(M),\nu^*(\gamma))$. Let us denote this bundle map by $s_0$ with abuse of notation. We can restrict the bundle map $s_0$ to $M-{\bf z}$ for every configuration of points ${\bf z}\in C_k(M)$ to obtain a bundle map in $\mathrm{Bun}_{\partial M}(T(M-{\bf z}),\nu^*(\gamma))$. Therefore, the map that sends ${\bf z}$ to the pair $({\bf z}, {s_0}|_{M-{\bf z}})$ gives a section for $\pi$.
 \end{proof}
 As was pointed out in the introduction, \Cref{punctures} is a consequence  of \Cref{surfaces} for $S=\bD^2$.  Hence we prove the latter using configuration bundle maps:
 \begin{proof}[Proof of \Cref{surfaces}] Recall that for a $S$ and ${\bf z}\in C_k(S)$, the group $\Diff(S,{\bf z})$ is the group of $C^{\infty}$-diffeomorphisms of $S$ whose supports are away from the boundary $\partial S$ and fix the points ${\bf z}$ as a set (they might permute points). It is a consequence of smoothing theory (\cite{MR0356103}) in dimension $2$ that the inclusion $\Diff(S,{\bf z})\hookrightarrow \Diff(S-{\bf z},\partial S)$ is a homotopy equivalence. Therefore the induced map between the following homotopy quotients of bundle maps
 \[
 \mathrm{Bun}_{\partial S}(T(S-{\bf z}),\nu^*(\gamma))\hcoker \Diff(S,{\bf z})\to  \mathcal{M}_{\nu}(S-{\bf z}),
 \]
 is in fact a weak equivalence. For brevity, in above weak equivalence we denote the first homotopy quotient by $\mathcal{M}_{\nu}(S,{\bf z})$. Using Mather-Thurston's theorem as we explained in \Cref{foliation}, we have a homotopy  commutative diagram 
 \begin{equation}\label{diagram1}
 \begin{tikzpicture}[node distance=3.9cm, auto]
  \node (A)  {$\BdDiff(S-{\bf z},\partial S)$};
  \node (C) [right of= A] {$\mathcal{M}_{\nu}(S-{\bf z})$};  
      \node (B) [below of= A, node distance=1.9 cm]{$\BDiff(S-{\bf z},\partial S)$};
    \node (D) [right of= C] {$\mathcal{M}_{\nu}(S,{\bf z})$};  
      \node (E) [right of= B] {$\mathrm{BMod}(S,{\bf z})$};
  \node (F) [right of= E] {$\mathrm{BBr}_{|{\bf z}|}(S).$};
   \draw [->] (B) to node {$$}(E);
   \draw [<-] (E) to node {$$}(F);
   \draw [<-] (B) to node {$$}(A);
  \draw [->] (C) to node {$$}(B);
  \draw [->] (A) to node {$H^*-\text{iso}$}(C);
    \draw [<-] (C) to node {$\simeq$}(D);
  \draw [->] (D) to node {$$}(E);
    \draw [->, dotted] (F) to node {$\alpha$}(D);
\end{tikzpicture}
 \end{equation}
Therefore to finish the proof, it suffices to prove that the dotted arrow $\alpha$ exists that makes the right most triangle commute up to homotopy. 

It is well known that for a surface $S$ that either has  boundary or is a closed orientable surface whose genus is positive, the configuration space $C_{|{\bf z}|}(S)$ is an aspherical space. Since its fundamental group is by definition $\mathrm{Br}_{|{\bf z}|}(S)$, the configuration space $C_{|{\bf z}|}(S)$ is a model for $\mathrm{BBr}_{|{\bf z}|}(S)$ for such surface $S$.

On the other hand, if $H$ is a subgroup of $G$ and the group $H$ acts on a space $X$, then we have the natural map $G\times_H X\rightarrow X\hcoker H$. Hence for $H=\Diff(S, {\bf z})$, $G=\Diff(S,\partial S)$ and $X= \mathrm{Bun}_{\partial S}(T(S-{\bf z}),\nu^*(\gamma))$, we obtain a natural map
\[
\mathrm{CBun}_k(S)\rightarrow \mathcal{M}_{\nu}(S,{\bf z}).
\]
The naturality of the above map implies that we have the following homotopy commutative diagram
\begin{equation}\label{ff}
 \begin{tikzpicture}[node distance=3.9cm, auto]
  \node (A) {$ \mathrm{CBun}_k(S)$};
  \node (B) [right of=A] {$\mathcal{M}_{\nu}(S,{\bf z})$};
  \node (C) [below of= A, node distance=1.8cm ] {$\mathrm{BBr}_{|{\bf z}|}(S)$};  
  \node (D) [right of=C] {$\mathrm{BMod}(S, {\bf z}).$};
   \draw [->] (A) to node {$$}(B);
  \draw [->] (C) to node {$$}(D);
  \draw [->] (A) to node {$\pi$}(C);
  \draw [->] (B) to node {$$} (D);
  \draw[->,dotted] (C) to node {$\alpha$} (B);
\end{tikzpicture}
\end{equation}
Recall by \Cref{section}, the projection $\pi$ has a section, therefore the dotted arrow $\alpha$ exists that makes the bottom triangle commute up to homotopy.
 \end{proof}
 \begin{rem}
 Since we do not know the analogue of Mather-Thurston's theorem (\Cref{foliation}) for diffeomorphisms with marked points, we still do not know if the map
 \[
H^*(\mathrm{Mod}(\bD^2, {\bf z});A)\to H^*(\dDiff(\bD^2,{\bf z});A)
\]
is split injective in all cohomological degrees and for all abelian groups $A$. But Salter-Tshishiku (\cite{salter2015nonrealizability}) used Thurston's stability theorem (\cite[Theorem 2]{thurston1974generalization}) to show that in fact the projection 
\[
\dDiff(\bD^2,{\bf z})\to \mathrm{Mod}(\bD^2, {\bf z}),
\]
does not admit a section. Hence it is still unknown if the existence of a section has a cohomological obstruction in this case.
 \end{rem}
 \begin{rem}
 There are other cases  that similar statements as \Cref{punctures} holds, for example the map $H^*(\BDiff(M);A)\to H^*(\BdDiff(M);A)$ is known to be injective for $M=S^1, S^3$ and any hyperbolic three manifolds (see \cite{nariman2016powers} for more details). 
 \end{rem}
 \begin{rem}\label{a} By the theorem of Hatcher $\Diff(\bD^3,\partial \bD^3)$ is contractible (\cite{hatcher1983proof}), thus using smoothing theory and the fibration \ref{fibration}, we deduce that $\BDiff(\bD^3-{\bf z}, \partial \bD^3)\simeq C_{|{\bf z}|}(\bD^3)$. Therefore a similar argument shows that 
\[
H^*(\BDiff(\bD^3-{\bf z}, \partial \bD^3);A)\hookrightarrow H^*(\BdDiff(\bD^3-{\bf z}, \partial \bD^3);A)
\]
is also injective. But unlike the punctured $2$ disk, $\BDiff(\bD^3-{\bf z}, \partial \bD^3)\nsimeq \mathrm{BMod}(\bD^3,{\bf z})$ where the mapping class group $\mathrm{Mod}(\bD^3,{\bf z})$ is in fact isomorphic to the permutation group on $|{\bf z}|$ letters. 
\end{rem}
\begin{rem}
One can in fact show that for a surface $S$ with boundary $H_i(\BdDiff(S-{\bf z},\partial S);\bZ)$ is independent of the number of the points $|{\bf z}|$ while $i\ll{|\bf z|}$. We will not pursue this homological stability phenomenon in this paper.
\end{rem}
\subsection{Translating \Cref{punctures} to the language of bordism} \label{sec:4} We are often interested in understanding the homology of a discrete group $G$ namely infinite symmetric group, braid groups on infinite number of strands or compactly supported diffeomorphisms of $\bR^n$. In desirable cases, $G$ is either perfect or has a canonical perfect subgroup. Therefore it makes sense to take the Quillen plus construction of $\mathrm{B}G$. The plus construction of $\mathrm{B}G$  is more amenable to homotopy theory and it is often weakly equivalent to an iterated loop space.  Fuks in \cite{fuks1974quillenization} found a differential-topological formulation for few results of this type. Roughly, the homotopy theoretical object that is homology equivalent to $\mathrm{B}G$ or weakly equivalent to its plus construction, classifies a certain structure on smooth manifolds. Among such structures, there are naturally selected ``non-singular" structures that have geometric significance and are classified by $\mathrm{B}G$. Hence, the homology equivalence of the space $\mathrm{B}G$ to the homotopy theoretical object can be translated as the manifolds with specific structures are bordant to  manifolds with  ``non-singular" such structures. 

In this section, we give a differential topological meaning to \Cref{punctures}. Let us first recall the bordism formulation of G.Segal's theorem (\cite[Theorem 3]{segal1973configuration}) which in a special case says that there is a map
\begin{equation}\label{scanningmap}
\phi:\mathrm{BMod}(\bD^2, {\bf z})\to \Omega^2\bS^2
\end{equation}
that is homology isomorphism onto the connected component that it hits while the homological degree is less than $n/2$. We briefly describe the map $\phi$. The appropriate model for $\mathrm{BMod}(\bD^2, {\bf z})$ is the configuration space $C_n(\bR^2)$ of $|{\bf z}|=n$ unordered points in $\bR^2$. We identify $\bS^2$ with the one-point compactification of $\bR^2$. To every configuration of points $\xi\in C_n(\bR^2)$, we associate a map $\phi(\xi): (\bS^2,\infty)\to (\bS^2,\infty)$ which sends the point at infinity to itself. We surround each point of $\xi$ by a ball of radius $d/3$ where $d$ is the minimum distance between distinct points in the set $\xi$. The map $\phi(\xi)$ is the map $\bR^2\to \bS^2$ that sends the complement of the balls around points in $\xi$ to the base point $\infty$ and on each ball is the canonical map $\bD^2\to \bS^2$ which collapses the boundary to the point at infinity. The map $\phi$ lands in $\Omega_n^2\bS^2$ the degree $n$ component of $\Omega^2\bS^2$, but since $\Omega^2\bS^2$ is an H-space all of its components  are in fact homotopy equivalent.

In fact Segal showed that similarly defined  map $\phi: C_n(\bR^k)\to \Omega^k\bS^k$ as above induces a homology isomorphism onto the connected component that it hits in the same range of degrees as above. But we are more interested in the case $k=2$, because $C_n(\bR^k)$ is a $K(G,1)$ space for $k=2$ where $G=\mathrm{Br}_n(\bD^2)$.  
\begin{defn}
 Let a {\it $B_k$-object} be a triple $(M, N, \epsilon)$, where $M$ is an oriented compact smooth manifold, $N$ is an oriented smooth compact submanifold of the product $M \times \text{int}(\bD^k)$ of dimension equal to the dimension of the manifold $M$, and $\epsilon$ is a trivialization of the normal bundle to $N$ in $M\times \text{int}(\bD^k)$. We say $(M, N, \epsilon)$ is ``non-singular" if the restriction of the projection map $ p:M\times \text{int}(\bD^k) \to M$ to $N$ is a regular map i.e. at every point of $N$ the derivative of $p|_N$ is non-degenerate. We say $(M,N,\epsilon)$ is bordant to $(M',N',\epsilon')$ if there exists another $B_k$ object $(W,V,\eta)$ whose boundary is $(M\coprod -M', N\coprod -N', \epsilon\coprod \epsilon')$.
 \end{defn}
 Fuks in \cite{fuks1974quillenization} proved that the Segal's theorem (\cite[Theorem 3]{segal1973configuration}) is equivalent to
 \begin{thm}[Fuks 1974]\label{Fuks}
 Every $B_k$-object which does not have boundary is bordant to a nonsingular $B_k$-object. Bordant nonsingular $B_k$-objects can be connected by a nonsingular bordism.
 \end{thm}
 \begin{proof}[Sketch]
 By applying the oriented bordism functor $\mathrm{MSO}_*$ to the  map $\phi$, we know by Atiyah-Hirzebruch spectral sequence that 
 \[
\phi_i: \mathrm{MSO}_i(C_n(\bD^k))\to \mathrm{MSO}_i(\Omega_n^k\bS^k)
 \]
 is isomorphism as long as $n\geq 2i$. An element in  $\mathrm{MSO}_i(\Omega_n^k\bS^k)$ can be thought of as a map $f: M\times \bS^k\to \bS^k$ for some $i$-dimensional manifold $M$ where $f(M\times \infty)=\infty$. We can assume that $f$ is smooth and is transverse to zero, hence $f^{-1}(0)=N\subset M\times \bS^k$ is a codimensioon $k$ submanifold with a trivial normal bundle. Let $\epsilon$ be the trivialization obtained by the canonical frames at zero in $\bS^k$. Since $\phi_i$ is an isomorphism for $n$ large, there exists a manifold $W$ of dimension $i+1$ and a map $G: W\to \Omega^k\bS^k$ such that $G(\partial W\backslash M)\subset \phi(C_n(\bD^k))$. One can assume that the map $g:W\times \bS^k\to \bS^k$ which is the adjoint to $G$ is transverse to zero. Similarly, we can choose a canonical trivialization $\eta$ for $V=g^{-1}(0)$. It is easy to see that $(W,V,\eta)$ is a bordism between $(M,N,\epsilon)$ and a non-singular triple $(M',N',\epsilon')$. 
 \end{proof}

 Recall from \Cref{foliation} and \Cref{a} that we know the map in 
 \begin{equation}\label{d}
 \BdDiff(\bD^k-{\bf z},\partial\bD^k)\to \mathrm{CBun}_{|{\bf z}|}(\bD^k),
 \end{equation}
 is a homology equivalence for $k=2$ and $k=3$. To interpret this homology isomorphism from differential topological point of view, we  define the following obejcts:   
 \begin{defn}
Let a {\it $C_k$-object} be a quadruple  $(M, N, \epsilon,\mathcal{F})$, where   $(M, N, \epsilon)$ is a $B_k$-object and $\mathcal{F}$ is a $\mathrm{S}\Gamma_k$-structure with a trivial normal bundle on $M\times \bD^k\backslash N$ such that near the boundary of the fibers of $M\times \bD^k\backslash N\to M$ coincides with the horizontal foliation (i.e. foliation given by pull back of the point foliation via $M\times \bD^k\to \bD^k$). A $C_k$-object is nonsingular if the $\mathrm{S}\Gamma_k$-structure structure comes from a codimension $k$ foliation transverse to the fiber of $M\times \bD^k\backslash N\to M$.
 \end{defn}
 Similar to Fuks' theorem, we can translate the homology isomorphism induced by the map in \ref{d} to the language of bordism as follows
 \begin{prop}
For $k=2$ and $k=3$, every $C_k$-object which does not have boundary is bordant to a nonsingular $C_k$-object. Bordant nonsingular $C_k$-objects can be connected by a nonsingular bordism.
 \end{prop}
 \begin{proof}
Recall that for a manifold $M$, homotopy classes of maps $\pi:M\to C_{|{\bf z}|}(\bD^k)$ are in bijection with the bordism class of $B_k$-objects $(M,N,\epsilon)$ such that $N\to M$ is a $|{\bf z}|$-covering map. Note that for a $B_k$-object $(M,N,\epsilon)$, the projection $p:M\times \bD^k\backslash N\to M$ is a fiber bundle whose fibers are diffeomorphic to $\bD^k-{\bf z}$ where $|{\bf z}|$ is the degree of the covering map $N\to M$. Now suppose the classifying map $\pi$ lifts to the configuration bundle maps
 \[
 \begin{tikzpicture}[node distance=1.8cm, auto]
  \node (A) {$M$};
  \node (B) [right of=A] {$C_{|{\bf z}|}(\bD^k),$};
  \node (C) [above of= B ] {$\mathrm{CBun}_{|{\bf z}|}(\bD^k)$};  
   \draw [->] (A) to node {$\pi$}(B);
  \draw [->] (C) to node {$$}(B);
  \draw [->, dotted] (A) to node {$$}(C);
\end{tikzpicture}
 \]
then the bundle structure on the vertical tangent bundle of $p:M\times \bD^k\backslash N\to M$ lifts to $\mathrm{BS}\Gamma_k$. Thus there exists a $\mathrm{S}\Gamma_k$-structure $\mathcal{F}$ on $M\times \bD^k\backslash N$. Hence an element in the bordism group $\mathrm{MSO}_*( \mathrm{CBun}_{|{\bf z}|}(\bD^k))$ consists of a data $(M,N,\epsilon, \mathcal{F})$ where $N\subset M\times \bD^k$ has a trivial normal bundle with the trivialization $\epsilon$ such that the projection map $N\to M$ is regular (covering map in this case whose degree is $|{\bf z}|$) at every point in $N$ and $\mathcal{F}$ is a codimension $k$ Haefliger structure on $M\times \bD^k\backslash N$ that coincides with the horizontal foliation near the boundary of the fibers $M\times \bD^k\to M$.  From \Cref{foliation} and \Cref{a}, we know that 
 \begin{equation}\label{eq:1}
 \mathrm{MSO}_*(\BdDiff(\bD^k-{\bf z},\partial\bD^k))\to \mathrm{MSO}_*( \mathrm{CBun}_{|{\bf z}|}(\bD^k))
 \end{equation}
 is an isomorphism which means that $(M,N,\epsilon, \mathcal{F})$ is cobordant to $(M',N',\epsilon',\mathcal{F}')$ where $N'\subset M'\times \bD^k$, the projection $N'\to M'$ is $|{\bf z}|$-cover and $\mathcal{F}$ is a foliation on the fiber bundle $ M'\times \bD^k\backslash N'\to M'$ which is transverse to the fibers.
  \end{proof}

Therefore, we can translate the splitting theorem \ref{punctures} to
\begin{thm}
For $k=2$ and $k=3$, every $B_k$-object $(M,N,\epsilon)$ is bordant to a nonsingular $B_k$-object $(M',N',\epsilon')$ where the fiber bundle $M'\times \bD^k\backslash N'\to M'$ given by the projection to $M'$ admits a foliation transverse to the fibers. 
\end{thm}
\subsection{Inducing up the map $\mathrm{BBr}_{2g+2}(\bD^2)\to \mathrm{BMod}(\Sigma_{g,2})$}\label{induceup} Recall that from the introduction that there is a map
\begin{equation}\label{phi}
\phi: {C}_{2g+2}(\bD^2)\to \mathcal{M}_{g,2}
\end{equation}
 which sends $2g+2$ points ${\bf z}= \{z_1,z_2,\dots, z_{2g+2}\}$ to the Riemann surface $\Sigma_{{\bf z}}$ given by 
 \[
 f_{{\bf z}}(z)^2 =\prod_{i=1}^{2g+2} (z-z_i)
 \]
 which is a branch cover over the unit disk with ${\bf z}$ as branched points. Let $\pi_{\bf z}: \Sigma_{{\bf z}}\to \bD^2$ denote the branched covering map. Note that the map $\phi$ induces a group homomorphism between the fundamental groups
 \[
 \psi: \mathrm{Br}_{2g+2}(\bD^2)\to \mathrm{Mod}(\Sigma_{g,2}).
 \]
First let us observe that the homomorphism $\psi$ can be lifted to a group homomorphism 
 \[
 \psi':\dDiff(\bD^2, (2g+2)-\text{marked points})\to \mathrm{Homeo}(\Sigma_{g,2},\partial \Sigma_{g,2}).
 \]
 This is because by the covering space theory, we can lift every diffeomorphism of $\bD^2-{\bf z}$ to a diffeomorphism of $\Sigma_{\bf z}-\pi_{\bf z}^{-1}({\bf z})$ and we can extend  a diffeomorphism of  $\Sigma_{\bf z}-\pi_{\bf z}^{-1}({\bf z})$ over the ramification points to obtain a homeomorphism.
 
  But for this lift $\psi'$, it is easy to see that the image of $\psi'$ does not land in $\dDiff(\Sigma_{g,2},\partial \Sigma_{g,2}) $.  For example, let $f\in \dDiff(\bD^2,{\bf z})$ be a diffeomorphism that fixes $z_1$ and its derivative $Df_{z_1}$ in a coordinate is given by the linear transformation that sends  $(x,y)$ to $(2x,y)$. Any lift of $f$ to a homeomorphism of $\Sigma_{\bf z}$ cannot be differentiated at the ramification point above $z_1$. Because $\pi_{\bf z}$ is a degree $2$ map in a neighborhood of ramification points, if $D\psi'(f)$ exists at the ramification point above $z_1$, as a linear transformation in a coordinate should send $(x,0)$ to $(2x,0)$ and $(0,y)$ to $(0,2y)$ but it should fix pointwise the lines $x=y$ and $x=-y$ which is not possible. 
  
 However we show that there is no homological obstruction to lift $\psi$ to a group homomorphism $\Psi$ between diffeomorphism groups. Recall from \Cref{sec:2} that there exists a map
 \[
\BdDiff(\bD^2-{\bf z},\partial\bD^2)\to \mathrm{CBun}_{|{\bf z}|}(\bD^2),
 \]
which  induces a homology equivalence. Similarly for diffeomorphism groups of surfaces, we know (see \cite[Section 4.1.3]{nariman2015stable} for more details) that there is a map
 \[
 \BdDiff(\Sigma_{g,2},\partial\Sigma_{g,2})\to \mathrm{Bun}_{\partial \Sigma_{g,2}}(T(\Sigma_{g,2}),\nu^*(\gamma))\hcoker \Diff(\Sigma_{g,2},\partial\Sigma_{g,2})
 \]
 which is an isomorphism on homology. Recall from \Cref{foliation} that fixing an isomorphism between the tangent bundle $T(\Sigma_{g,2})$ and $\tau^*(\gamma)$, we can think of space of bundle maps $\mathrm{Bun}_{\partial \Sigma_{g,2}}(T(\Sigma_{g,2}),\nu^*(\gamma))$ as the space of lifts of the tangential structure of  $\Sigma_{g,2}$ to $\mathrm{BS}\Gamma_2$ 
   \[
 \begin{tikzpicture}[node distance=1.8cm, auto]
  \node (A) {$\Sigma_{g,2}$};
  \node (B) [right of=A] {$\mathrm{BGL}_{2}(\bR)$};
  \node (C) [above of= B ] {$\mathrm{BS}\Gamma_{2}$};  
   \draw [->] (A) to node {$\tau$}(B);
  \draw [->] (C) to node {$\nu$}(B);
  \draw [->, dotted] (A) to node {$f$}(C);
\end{tikzpicture}
 \]
 that are equal to the base section $s_0$ (induced by the point foliation) near the boundary $\partial \Sigma_{g,2}$. The goal in this section is to define a map
 \begin{equation}\label{eq:3}
\BdDiff(\bD^2,{\bf z})\to \mathrm{Bun}_{\partial \Sigma_{g,2}}(T(\Sigma_{g,2}),\nu^*(\gamma))\hcoker \Diff(\Sigma_{g,2},\partial\Sigma_{g,2})
 \end{equation}
 that naturally lives over $\phi$ in \ref{phi}.
 
 A model for the homotopy quotient $\mathrm{Bun}_{\partial \Sigma_{g,2}}(T(\Sigma_{g,2}),\nu^*(\gamma))\hcoker \Diff(\Sigma_{g,2},\partial\Sigma_{g,2})
$ is 
\[
\mathcal{H}(\Sigma_{g,2})\times \mathrm{Bun}_{\partial \Sigma_{g,2}}(T(\Sigma_{g,2}),\nu^*(\gamma))/\Diff(\Sigma_{g,2},\partial\Sigma_{g,2}),
\] 
where  $\mathcal{H}(\Sigma_{g,2})$  is the space of hyperbolic metrics on $\Sigma_{g,2}$. Note that the action of $\Diff(\Sigma_{g,2},\partial\Sigma_{g,2})$ on $\mathcal{H}(\Sigma_{g,2})$ is free. Since we have the following fiber bundle
\[
\mathrm{Bun}_{\partial \Sigma_{g,2}}(T(\Sigma_{g,2}),\nu^*(\gamma))\to \mathrm{Bun}_{\partial \Sigma_{g,2}}(T(\Sigma_{g,2}),\nu^*(\gamma))\hcoker \Diff(\Sigma_{g,2},\partial\Sigma_{g,2})\to \mathcal{M}_{g,2},
\]
we can think of the homotopy quotient $\mathrm{Bun}_{\partial \Sigma_{g,2}}(T(\Sigma_{g,2}),\nu^*(\gamma))\hcoker \Diff(\Sigma_{g,2},\partial\Sigma_{g,2})$ geometrically as the space of Riemann surfaces equipped with a lifting of the tangential structure from $\mathrm{BGL}_2(\bR)$ to $\mathrm{BS}\Gamma_2$. We use these models to prove

\begin{thm}\label{Phi}
There exists a map $\Phi$ which makes the following diagram homotopy commutative
   \[
 \begin{tikzpicture}[node distance=6.7cm, auto]
  \node (A) {$ \BdDiff(\bD^2, (2g+2) \text{ marked points})$};
  \node (B) [right of=A] {$\mathrm{Bun}_{\partial \Sigma_{g,2}}(T(\Sigma_{g,2}),\nu^*(\gamma))\hcoker \Diff(\Sigma_{g,2},\partial\Sigma_{g,2})$};
  \node (C) [below of= A, node distance=1.8cm ] {$\mathrm{BBr}_{2g+2}(\bD^2)$};  
  \node (D) [right of=C] {$\mathrm{BMod}(\Sigma_{g,2}).$};
   \draw [->] (A) to node {$\Phi$}(B);
  \draw [->] (C) to node {$\phi$}(D);
  \draw [->] (A) to node {$$}(C);
  \draw [->] (B) to node {$$} (D);
\end{tikzpicture}
\]

\end{thm}
\begin{proof}
Here we think of bundle maps as the space of lifts of the tangential structure to $\mathrm{BS}\Gamma_2$. Note that  $\BdDiff(\bD^2,{\bf z})$ maps to the configuration section space by
\[
\BdDiff(\bD^2,{\bf z})\to \BdDiff(\bD^2-{\bf z},\partial\bD^2)\to \mathrm{CBun}_{|{\bf z}|}(\bD^2).
\]
The image of this compositions lands in the subspace $$\mathrm{Bun}_{\partial \bD^2}(T(\bD^2),\nu^*(\gamma))\hcoker \Diff(\bD^2,{\bf z})\subset \mathrm{CBun}_{|{\bf z}|}(\bD^2),$$ in other words the image consists of pairs $({\bf a},g)$ where ${\bf a}$ is a configuration of $|{\bf z}|$ points in the disk and a section $g$ that is defined over the entire disk $\bD^2$.

 Let ${\bf z}\in C_{2g+2}(\bD^2)$ be a configuration of $2g+2$ points. The surface $\phi({\bf z})=\Sigma_{{\bf z}}$ is a branch double cover over the disk
\[
\pi_{\bf z}: \Sigma_{{\bf z}}\to \bD^2,
\]
which is branched over ${\bf z}$. Let $\pi':\Sigma_{{\bf z}}-\pi_{\bf z}^{-1}({\bf z})\to \bD^2-{\bf z}$ be the double cover on the complement of the branch points. Consider the diagram 
\[
 \begin{tikzpicture}[node distance=1.8cm, auto]
  \node (A) {$\bD^2-{\bf z}$};
  \node (D) [left of=A, node distance=2.3cm] {$\Sigma_{{\bf z}}-\pi_{\bf z}^{-1}({\bf z})$};
  \node (B) [right of=A] {$\mathrm{BGL}_{2}(\bR)$};
  \node (C) [above of= B ] {$\mathrm{BS}\Gamma_{2}$};  
  \node (E) [above of=A] {$\bD^2$};
  \node (F) [above of= D] {$\Sigma_{\bf z}$};
  \draw [right hook->] (D) to node {$$} (F);
    \draw [right hook->] (A) to node {$$} (E);
   \draw [->] (E) to node {$\tilde{f}$}(C);
      \draw [->] (F) to node {$\pi_{\bf z}$}(E);
   \draw [->] (A) to node {$\tau'$}(B);
  \draw [->] (C) to node {$\nu$}(B);
  \draw [->, dotted] (A) to node {$f$}(C);
  \draw [->] (D) to node {$\pi'$}(A);
\end{tikzpicture}
 \]
 where $\tau'$ classifies the tangent bundle of $\bD^2-{\bf z}$. Now since the image of the map
 \[
 \BdDiff(\bD^2,(2g+2) \text{ marked points})\to \mathrm{CBun}_{|{\bf z}|}(\bD^2),
  \]
consists of the pairs $({\bf z}, \tilde{f})$ where ${\bf z}$ is in $C_{2g+2}(\bD^2)$ and the restriction of $\tilde{f}$ to the punctured disk $\bD^2-{\bf z}$  is $f$. Hence,  we can define a map
 \[
\Phi: \BdDiff(\bD^2,(2g+2) \text{ marked points})\to\mathrm{Bun}_{\partial \Sigma_{g,2}}(T(\Sigma_{g,2}),\nu^*(\gamma))\hcoker \Diff(\Sigma_{g,2},\partial\Sigma_{g,2})
 \]
 that sends $({\bf z},f)$ to $(\Sigma_{\bf z}, \tilde{f}\circ\pi_{\bf z})$. Therefore, we obtain the homotopy commutative diagram of the theorem.
 \end{proof}
\begin{cor}\label{lift}
There exists a lift of $\phi$ that makes the following diagram homotopy commutative
   \[
 \begin{tikzpicture}[node distance=2cm, auto]
   \node (C) {$\mathrm{BBr}_{2g+2}(\bD^2)$};  
  \node (D) [right of=C, node distance=4.5cm] {$\mathrm{BMod}(\Sigma_{g,2}).$};
  \node (B) [above of=D] {$\mathrm{Bun}_{\partial \Sigma_{g,2}}(T(\Sigma_{g,2}),\nu^*(\gamma))\hcoker \Diff(\Sigma_{g,2},\partial\Sigma_{g,2})$};

  \draw [->] (C) to node {$\phi$}(D);
  \draw [->] (C) to node {$$}(B);
  \draw [->] (B) to node {$h$} (D);
\end{tikzpicture}
\]

\end{cor}
\begin{proof}
Similar to the proof of the above theorem, we first define a map 
\[
\alpha: C_{|\bf z|}(\bD^2)\to \mathcal{H}(\Sigma_{g,2})\times \mathrm{Bun}_{\partial \Sigma_{g,2}}(T(\Sigma_{g,2}),\nu^*(\gamma)),
\]
that sends a configuration of points ${\bf z}$, to the pair $(\Sigma_{\bf z}, s_0\circ \pi_{\bf z})$, where $s_0$ is the base section in $\mathrm{Bun}_{\partial \bD^2}(T\bD^2,\nu^*(\gamma))$ as in the diagram \ref{eq:11}. Hence, the map $\alpha$ induces a lift of $\phi$
\[
\mathrm{BBr}_{2g+2}(\bD^2)\to \mathrm{Bun}_{\partial \Sigma_{g,2}}(T(\Sigma_{g,2}),\nu^*(\gamma))\hcoker \Diff(\Sigma_{g,2},\partial\Sigma_{g,2}).
\]
\end{proof}
\begin{rem}
Given that $\mathrm{Bun}_{\partial \Sigma_{g,2}}(T(\Sigma_{g,2}),\nu^*(\gamma))\hcoker \Diff(\Sigma_{g,2},\partial\Sigma_{g,2})$ is homology equivalent to $\BdDiff(\Sigma_{g,2},\partial \Sigma_{g,2})$, the \Cref{lift} implies that for any abelian group $A$ the induced map on homology 
\[
\phi_*: H_*(\mathrm{BBr}_{2g+2}(\bD^2);A)\to H_*(\mathrm{BMod}(\Sigma_{g,2});A) ,
\]
factors through 
\[
h_*:  H_*(\BdDiff(\Sigma_{g,2},\partial \Sigma_{g,2});A)\to H_*(\mathrm{BMod}(\Sigma_{g,2});A).
\]
In \cite[Theorem 3.23]{nariman2015stable} for a finite field $A=\bF_p$, we showed that the map $h_*$ is surjective for $*\leq (2g-2)/3$. Given that Segal and Tillmann in \cite{segal2007mapping} proved that $\phi_*$ is trivial in the stable range, one can deduce that the map
\[
H_*(\mathrm{BBr}_{2g+2}(\bD^2);\bF_p)\to H_*(\BdDiff(\Sigma_{g,2},\partial \Sigma_{g,2});\bF_p) 
\]
is also trivial in the stable range $*\leq (2g-2)/3$.
 \end{rem}

\Cref{lift} implies that there is no cohomological obstruction to lift the map $\phi$ in the diagram 
\begin{equation}\label{eq:12}
\begin{gathered}
 \begin{tikzpicture}[node distance=2cm, auto]
   \node (C) {$\mathrm{BBr}_{2g+2}(\bD^2)$};  
  \node (D) [right of=C, node distance=4.5cm] {$\mathrm{BMod}(\Sigma_{g,2}).$};
  \node (B) [above of=D] {$\BdDiff(\Sigma_{g,2},\partial \Sigma_{g,2})$};

  \draw [->] (C) to node {$\phi$}(D);
   \draw [->] (B) to node {$h$} (D);
\end{tikzpicture}
\end{gathered}
\end{equation}
to $\BdDiff(\Sigma_{g,2},\partial \Sigma_{g,2})$. Inspired by the work of \cite{salter2015nonrealizability}, we show that in fact there is no such a lift.
\begin{thm}
For $g>1$, the map $\phi$ cannot be lifted to $\BdDiff(\Sigma_{g,2},\partial \Sigma_{g,2})$ in the diagram \ref{eq:12}.
\end{thm}
\begin{proof}Suppose the contrary. If we can lift $\phi$, then we have the commutative diagram by taking the fundamental group of the diagram \ref{eq:12}
\begin{equation}\label{eq:13}
\begin{gathered}
 \begin{tikzpicture}[node distance=2cm, auto]
   \node (C) {$\mathrm{Br}_{2g+2}(\bD^2)$};  
  \node (D) [right of=C, node distance=4.5cm] {$\mathrm{Mod}(\Sigma_{g,2}).$};
  \node (B) [above of=D] {$\dDiff(\Sigma_{g,2},\partial \Sigma_{g,2})$};

   \draw [dotted, ->] (A) to node {$g$}(B);
  \draw [->] (C) to node {$\psi$}(D);
   \draw [->] (B) to node {$p$} (D);
\end{tikzpicture}
\end{gathered}
\end{equation}
Since $\psi$ is injective so is the map $g$ and therefore $g$ in particular injects the commutator subgroup $[\mathrm{Br}_{2g+2}(\bD^2),\mathrm{Br}_{2g+2}(\bD^2)]$ into $\dDiff(\Sigma_{g,2},\partial \Sigma_{g,2})$. Let us denote this subgroup of $\dDiff(\Sigma_{g,2},\partial \Sigma_{g,2})$ by $H$. By the theorem of Gorin and Lin \cite{MR0251712}, the group $H$ is a finitely generated perfect group for $g>1$.

Now we use Thurston's stability theorem \cite{thurston1974generalization}  to get a contradiction. Thurston's stability theorem says that for a manifold $M$ and a point $x\in M$, the group $$C^1-\text{Stab}(x):=\{ f\in \dDiff(M): f(x)=x, \text{ and } Df_x=Id\},$$ is {\it locally indicable}, meaning that every finitely generated subgroup in $C^1-\text{Stab}(x)$ surjects to $\bZ$. Let $M\cong\Sigma_{g,2}$ and $x\in \partial \Sigma_{g,2}$. Since every element of $\dDiff(\Sigma_{g,2},\partial \Sigma_{g,2})$ fixes a neighborhood of the boundary, the group $H$ is a subgroup of $C^1-\text{Stab}(x)$. But since $H$ is perfect, there is no nontrivial homomorphism from $H$ to $\bZ$, which is a contradiction.
\end{proof}
To show that the $\Phi$ in \Cref{Phi} is a nontrivial map, we prove it induces a  nontrivial map on homology.  Given that $\mathrm{Bun}_{\partial \Sigma_{g,2}}(T(\Sigma_{g,2}),\nu^*(\gamma))\hcoker \Diff(\Sigma_{g,2},\partial\Sigma_{g,2})$ is homology equivalent to $\BdDiff(\Sigma_{g,2},\partial \Sigma_{g,2})$, the induced map on homology is
\[
H_*( \BdDiff(\bD^2, (2g+2) \text{ marked points});\bZ)\to H_*(\BdDiff(\Sigma_{g,2},\partial\Sigma_{g,2});\bZ).
\]
Unlike the theorem of Tillmann and Segal \cite[Corollary 4.3]{segal2007mapping} that proved that $\phi$ induces a trivial map 
\[
H_*(\mathrm{Br}_{2g+2}(\bD^2);\bZ)\to H_*(\mathrm{Mod}(\Sigma_{g,2});\bZ)
\]
in the stable range, we show
\begin{thm}
For all $g$, the induced map by $\Phi$ on the third homology
\[
H_3( \BdDiff(\bD^2, (2g+2) \text{ marked points});\bZ)\to H_3(\BdDiff(\Sigma_{g,2},\partial\Sigma_{g,2});\bZ).
\]
is nonzero.
\end{thm}
\begin{proof}
Embed $\bR^2$ as the interior of a small disk into $\bD^2-{\bf z}$ so that it lifts to two disjoint disks in $\Sigma_{\bf z}$ under the map $\pi':\Sigma_{\bf z}-\pi_{\bf z}^{-1}({\bf z})\to \bD^2-{\bf z}$. This embedding induces the following commutative diagram
\[
 \begin{tikzpicture}[node distance=4.3cm, auto]
  \node (A) {$ \BdDiff(\bD^2, {\bf z})$};
    \node (C) [below of= A, node distance=1.8cm ] {$\BdDiff(\Sigma_{\bf z},\partial\Sigma_{\bf z})$};  
  \node (E) [left of=C, node distance= 3.5cm]{$\BdDiff_c(\bR^2)$};
  \node (D) [right of=C, node distance=5.2cm] {$\mathrm{Bun}_{\partial \Sigma_{g,2}}(T(\Sigma_{g,2}),\nu^*(\gamma))\hcoker \Diff(\Sigma_{\bf z},\partial\Sigma_{\bf z}).$};
   \draw [->] (A) to node {$$}(D);
  \draw [->] (C) to node {$\stackrel{H_*}{\cong}$}(D);
  \draw [->] (E) to node {$$}(C);
    \draw [->] (E) to node {$$}(A);
\end{tikzpicture}
\]
Therefore, we have the following commutative diagram on homology groups
\begin{equation}\label{k}
\begin{gathered}
 \begin{tikzpicture}[node distance=4.3cm, auto]
  \node (A) {$H_3( \BdDiff(\bD^2, {\bf z});\bZ)$};
  \node (E) [left of=C]{$H_3(\BdDiff_c(\bR^2);\bZ)$};
  \node (C) [below of= A, node distance=1.8cm ] {$H_3(\BdDiff(\Sigma_{\bf z},\partial\Sigma_{\bf z});\bZ),$};  
  \draw [->] (E) to node {$k$}(C);
  \draw [->] (A) to node {$$} (C);
    \draw [->] (E) to node {$$}(A);
\end{tikzpicture}
\end{gathered}
\end{equation}
so if we show the  map $k$ is nonzero, we are done.  We use continuous variation of Godbillon-Vey classes to show that $k$ induces a nontrivial map on homology with rational coefficients. There are two universal Godbillon-Vey classes $h_1c_2$ and $ h_1c_1^2$ for codimension two foliations (see \cite[Chapter 2]{pittie1976characteristic} for the definition of these classes) that live in $H^5(\mathrm{BS}\Gamma_2;\bR)$. Therefore for any codimension $2$ Haefliger structure $\mathcal{F}$ on a space $X$, we have two classes $h_1c_2$ and $ h_1c_1^2$ in $H^5(X;\bR)$ associated to $\mathcal{F}$. Now consider the universal flat surface bundle

\[
 \begin{tikzpicture}[node distance=2.3cm, auto]
  \node (A) {$\Sigma_{\bf z}$};
  \node (E) [right of=A]{$\Sigma_{\bf z}\hcoker \dDiff(\Sigma_{\bf z},\partial \Sigma_{\bf z})$};
  \node (C) [below of= E, node distance=1.5cm ] {$\BdDiff(\Sigma_{\bf z},\partial\Sigma_{\bf z}).$};  
  \draw [->] (A) to node {$$}(E);
    \draw [->] (E) to node {$$}(C);
\end{tikzpicture}
\] Since the bundle is flat there is a codimension $2$ Haefliger structure on the total space (that is in fact transverse to the fibers), the characteristic classes associated to this structure live in $H^5(\Sigma_{\bf z}\hcoker \dDiff(\Sigma_{\bf z},\partial \Sigma_{\bf z});\bR)$. If we integrate these two classes on this flat $\Sigma_{\bf z}$-bundle, we obtain a map
\[
H_3(\BdDiff(\Sigma_{\bf z},\partial\Sigma_{\bf z});\bQ)\xrightarrow{(\int h_1c_2,\int h_1c_1^2)} \bR^2.
\]
Note that in the diagram \ref{k},  the map $k$ is induced by embedding two disjoint disks into the the surface, hence if we compose $k$ with the above map, we obtain
\[
H_3(\BdDiff_c(\bR^2);\bQ)\xrightarrow{(2\int h_1c_2,2\int h_1c_1^2)} \bR^2.
\]
Hence, we only need to show $H_3(\BdDiff_c(\bR^2);\bQ)\xrightarrow{(\int h_1c_2,\int h_1c_1^2)} \bR^2$ is nontrivial. Let $\overline{\BDiff_c(\bR^2)}$ be the homotopy fiber of the map
\[
\BdDiff_c(\bR^2)\to \BDiff_c(\bR^2).
\]
But the topological group $\Diff_c(\bR^2)$ is contractible, so $\overline{\BDiff_c(\bR^2)}\simeq \BdDiff_c(\bR^2)$. By  Thurston's theorem \cite{thurston1974foliations}, we know that there is a map 
\[
\overline{\BDiff_c(\bR^2)}\to \Omega^2\overline{\mathrm{BS}\Gamma_2}
\]
that induces a homology isomorphism. By the theorem of Haefliger and Thurston \cite[Theorem 10.3]{bott1972lectures}, we know $\overline{\mathrm{BS}\Gamma_2}$ is at least $3$-connected. Therefore, we have
\[
H_3(\BdDiff_c(\bR^2);\bQ)\xrightarrow{\cong}H_3( \Omega^2\overline{\mathrm{BS}\Gamma_2};\bQ)\twoheadrightarrow H_5(\overline{\mathrm{BS}\Gamma_2};\bQ) \twoheadrightarrow \bR^2.
\]
The first map is an isomorphism by  Thurston's theorem. The second map is the suspension map which  is surjective by the rational Hurewicz theorem.  The third map is given by Godbillon-Vey classes
\[
H_5(\overline{\mathrm{BS}\Gamma_2};\bQ) \xrightarrow{(\int h_1c_2,\int h_1c_1^2)} \bR^2
\]
which is surjective as a corollary of the theorem of Rasmussen \cite{rasmussen1980continuous}.
\end{proof}
\bibliographystyle{alpha}
\bibliography{reference}

\begin{thebibliography}{EVW12}

\bibitem[AS12]{aramayona2012rigidity}
Javier Aramayona and Juan Souto.
\newblock Rigidity phenomena in the mapping class group.
\newblock {\em preprint}, 2012.

\bibitem[BH73]{birman1973isotopies}
Joan~S Birman and Hugh~M Hilden.
\newblock On isotopies of homeomorphisms of {R}iemann surfaces.
\newblock {\em Annals of Mathematics}, pages 424--439, 1973.

\bibitem[BL74]{MR0356103}
Dan Burghelea and Richard Lashof.
\newblock The homotopy type of the space of diffeomorphisms. {I}, {II}.
\newblock {\em Trans. Amer. Math. Soc.}, 196:1--36; ibid. 196\ (1974), 37--50,
  1974.

\bibitem[Bot72]{bott1972lectures}
Raoul Bott.
\newblock {\em Lectures on characteristic classes and foliations}.
\newblock Springer, 1972.

\bibitem[EVW12]{ellenberg2012homological}
Jordan~S Ellenberg, Akshay Venkatesh, and Craig Westerland.
\newblock Homological stability for {H}urwitz spaces and the {C}ohen-{L}enstra
  conjecture over function fields, ii.
\newblock {\em arXiv:1212.0923}, 2012.

\bibitem[Fuk74]{fuks1974quillenization}
DB~Fuks.
\newblock Quillenization and bordism.
\newblock {\em Functional Analysis and Its Applications}, 8(1):31--36, 1974.

\bibitem[GL69]{MR0251712}
E.~A. Gorin and V.~Ja. Lin.
\newblock Algebraic equations with continuous coefficients, and certain
  questions of the algebraic theory of braids.
\newblock {\em Mat. Sb. (N.S.)}, 78 (120):579--610, 1969.

\bibitem[Hae71]{haefliger1971homotopy}
Andr{\'e} Haefliger.
\newblock Homotopy and integrability.
\newblock In {\em Manifolds-Amsterdam 1970}, pages 133--163. Springer, 1971.

\bibitem[Hat83]{hatcher1983proof}
Allen~E Hatcher.
\newblock A proof of the {S}male conjecture, {{$\text{Diff}(S^3)\simeq O(4)$}}.
\newblock {\em Annals of Mathematics}, pages 553--607, 1983.

\bibitem[Ker83]{kerckhoff1983nielsen}
Steven~P Kerckhoff.
\newblock The {N}ielsen realization problem.
\newblock {\em Annals of Mathematics}, pages 235--265, 1983.

\bibitem[Man17]{Katie}
Kathryn Mann.
\newblock Realizing maps of braid groups by surface diffeomorphisms.
\newblock {\em to appear}, 2017.

\bibitem[Mar07]{markovic2007realization}
Vladimir Markovic.
\newblock Realization of the mapping class group by homeomorphisms.
\newblock {\em Inventiones mathematicae}, 168(3):523--566, 2007.

\bibitem[Mat11]{mather2011homology}
John~N Mather.
\newblock On the homology of haefliger's classifying space.
\newblock In {\em Differential Topology}, pages 71--116. Springer, 2011.

\bibitem[McD80]{mcduff1980homology}
Dusa McDuff.
\newblock The homology of some groups of diffeomorphisms.
\newblock {\em Commentarii Mathematici Helvetici}, 55(1):97--129, 1980.

\bibitem[McD83]{mcduff1983local}
Dusa McDuff.
\newblock Local homology of groups of volume-preserving diffeomorphisms. iii.
\newblock In {\em Annales scientifiques de l'{\'E}cole Normale Sup{\'e}rieure},
  volume~16, pages 529--540. Soci{\'e}t{\'e} math{\'e}matique de France, 1983.

\bibitem[Mor87]{morita1987characteristic}
Shigeyuki Morita.
\newblock Characteristic classes of surface bundles.
\newblock {\em Inventiones mathematicae}, 90(3):551--577, 1987.

\bibitem[Nar14]{nariman2014homologicalstability}
Sam Nariman.
\newblock Homological stability and stable moduli of flat manifold bundles.
\newblock {\em arXiv:1406.6416}, 2014.

\bibitem[Nar15]{nariman2015stable}
Sam Nariman.
\newblock Stable homology of surface diffeomorphism groups made discrete.
\newblock {\em To appear in Geometry and Topology}, 2015.

\bibitem[Nar16]{nariman2016powers}
Sam Nariman.
\newblock On powers of the {E}uler class for flat circle bundles.
\newblock {\em Journal of Topology and Analysis}, pages 1--6, 2016.

\bibitem[Pit76]{pittie1976characteristic}
Harsh~V Pittie.
\newblock {\em Characteristic classes of foliations}.
\newblock Pitman, 1976.

\bibitem[Ras80]{rasmussen1980continuous}
Ole~Hjorth Rasmussen.
\newblock Continuous variation of foliations in codimension two.
\newblock {\em Topology}, 19(4):335--349, 1980.

\bibitem[Seg73]{segal1973configuration}
Graeme Segal.
\newblock Configuration-spaces and iterated loop-spaces.
\newblock {\em Inventiones Mathematicae}, 21(3):213--221, 1973.

\bibitem[Sma59]{MR0112149}
Stephen Smale.
\newblock Diffeomorphisms of the {$2$}-sphere.
\newblock {\em Proc. Amer. Math. Soc.}, 10:621--626, 1959.

\bibitem[ST07]{song2007braids}
Yongjin Song and Ulrike Tillmann.
\newblock Braids, mapping class groups, and categorical delooping.
\newblock {\em Mathematische Annalen}, 339(2):377--393, 2007.

\bibitem[ST08]{segal2007mapping}
Graeme Segal and Ulrike Tillmann.
\newblock Mapping configuration spaces to moduli spaces.
\newblock In {\em Groups of diffeomorphisms}, volume~52 of {\em Adv. Stud. Pure
  Math.}, pages 469--477. Math. Soc. Japan, Tokyo, 2008.

\bibitem[ST16]{salter2015nonrealizability}
Nick Salter and Bena Tshishiku.
\newblock On the non-realizability of braid groups by diffeomorphisms.
\newblock {\em Bull. Lond. Math. Soc.}, 48(3):457--471, 2016.

\bibitem[Thu74a]{thurston1974foliations}
William Thurston.
\newblock Foliations and groups of diffeomorphisms.
\newblock {\em Bulletin of the American Mathematical Society}, 80(2):304--307,
  1974.

\bibitem[Thu74b]{thurston1974generalization}
William~P Thurston.
\newblock A generalization of the {R}eeb stability theorem.
\newblock {\em Topology}, 13(4):347--352, 1974.

\bibitem[Thu11]{mathoverflow}
William Thurston.
\newblock Realizing the braid group by homeomorphism.
\newblock {\em Mathoverflow}, February 2011.

\end{thebibliography}
\end{document}